\documentclass[12pt,english]{amsart}
\usepackage{ae,aecompl}

\usepackage[T1]{fontenc}
\usepackage[latin9]{inputenc}
\usepackage{listings}
\usepackage{geometry}
\usepackage{babel}
\usepackage{array}
\usepackage{amsthm}
\usepackage{amssymb}

\usepackage{verbatim}
\usepackage{upquote}

\numberwithin{equation}{section}
\numberwithin{figure}{section}
\theoremstyle{plain}
\newtheorem{thm}{Theorem}
  \theoremstyle{remark}
  \newtheorem{rem}[thm]{Remark}
  \newtheorem{alg}{Algorithm}
\newtheorem*{rem*}{Remark}
  \theoremstyle{plain}
  \newtheorem{lem}[thm]{Lemma}

\begin{document}

\title{Computing a logarithm of a unitary matrix with general spectrum}

\author{Terry A. Loring}

\address{Department of Mathematics and Statistics, University of New Mexico,
Albuquerque, NM 87131, USA.}

\begin{abstract}
We analyze an algorithm for computing a skew-Hermitian logarithm
of a unitary matrix. This algorithm is very easy
to implement using standard software and it works well even for unitary
matrices with no spectral conditions assumed. Certain examples, with
many eigenvalues near $-1$, lead to very non-Hermitian
output for other basic methods of calculating matrix logarithms.
Altering the output of these algorithms to force an Hermitian
output creates accuracy issues which are avoided in the
considered algorithm.

A modification is introduced to deal properly with the
$J$-skew symmetric unitary
matrices.  Applications to numerical studies of
topological insulators in two symmetry classes are discussed.
\end{abstract}



\maketitle

\section{Introduction}

While all invertible matrices have a logarithm, indeed many logarithms,
the unitary matrices have the nicest logarithms. Every unitary matrix
$U$ has a skew-Hermitian logarithm. Indeed, when $U$ is unitary,
the conditions
\[
-\pi<K\leq\pi,\quad e^{iK}=U
\]
specify $K$ uniquely. Still working abstractly, we may take advantage
of the finiteness of the spectrum of $U$ when discussing additional
symmetries. For each $U$
there is a complex polynomial so that $K=p(U)$ which means, for example,
that when $U$ is $J$-skew-symmetric then so must be $K$.

We may regard $U^{*}=U^{-1}$ as a symmetry, but the non-linearity
of the unit circle inevitably makes this symmetry a little different
from that of being Hermitian ($A=A^{*}$) or symmetric ($A^{\top}=A$).
In the case of a matrix being Hermitian, if numerical errors start
to creep in, we simple replace $A$ by $\tfrac{1}{2}A+\tfrac{1}{2}A^{*}$
in ${\mathcal O}(n^{2})$ time and have again $A^{*}=A$ exactly.
The unitary part of the polar decomposition can be used when $U^{*}U$
is only close to $I$ but it takes more that ${\mathcal O}(n^{2})$ time
to compute and the result is not exactly unitary. For this reason we need to
consider errors in $\left\Vert U^{*}U-I\right\Vert $ a bit larger
than machine precision.

\begin{rem*}
We are using $U^{*}$ to denote conjugate-transpose of a matrix. This
would be denoted with a dagger in physics.  For just the conjugate
we use $\overline{U}$.

What is called {\em self-dual}
\cite{hastings2001eigenvalue} in physics is called {\em $J$-skew-symmetric}
\cite{bunseStructuredWigenvalue} or {\em skew Hamiltonian}
\cite{benner-skew} in computer science.

Finally, what
applied mathematicians call matrix functions are called, in pure math,
applications of the functional calculus.
\end{rem*}

The algorithm discussed here arose in a numerical study in condensed
matter physics \cite{HastLorTheoryPractice}. In that situation
$\left\Vert U^{*}U-I\right\Vert \approx\frac{1}{2}$
was typical. Such a matrix $U$ is still very well conditioned, so
one would expect almost any algorithm for computing a logarithm to
perform well, at least most of the time. However the study in question
had a time-reversal symmetry that, through a variation on Kramers
pairs, caused the approximate unitaries in question to have multiplicity
at least two in every eigenvalue. Otherwise the approximate unitaries
were free to have arbitrary spectrum within the unit circle. Some reflection
makes one realize that whatever branch of logarithm one uses, a degenerate
eigenvalue on the branch point has the potential to cause trouble.
This is true whether or not the degeneracy in the spectrum is caused
by a symmetry or by very bad luck.

We are not claiming it is hard to find $K$ with $e^{iK} \approx U$.
We are interested in achieving $e^{iK} \approx U$ and $K^*=K$ at
the same time.  Moreover we want an algorithm that can be easily
analyzed and also easily modified to take into account additional
symmetries.

To be consistent, we consider the branch of logarithm that has the
imaginary part of $\log(\lambda)$ in the interval $(-\pi,\pi]$, so in
particular $\log(-1)=i\pi$.
We are mainly concerned with the operator norm
\[
\left\Vert A\right\Vert 
=\sup_{\mathbf{v}\neq0}
\frac{\left\Vert A\mathbf{v}\right\Vert _{2}}
{\left\Vert \mathbf{v}\right\Vert _{2}}.
\]
We also
will use the Frobenius norm, i.e.~the un-normalized
Hilbert\textendash{}Schmidt norm
\[
\left\Vert A\right\Vert _{\mathrm{F}}
=\sqrt{\sum_{j,k=1}^{n}\left|A_{jk}\right|^{2}}.
\]
Recall the bounds 
$\left\Vert A\right\Vert 
\leq
\left\Vert A\right\Vert _{\mathrm{F}}
\leq
\sqrt{n}\left\Vert A\right\Vert $.

One easy solution, that often works well for logarithms of well-conditioned
matrices, starts by diagonalizing $U$ via an invertible.

\begin{alg} \ \label{alg:diagonalize}
\begin{enumerate}
\item Compute a diagonalization, $W$ invertible and $T$ diagonal with
$WTW^{-1} \approx U$.
\item Create a diagonal unitary matrix $D$ via $ $$D_{jj}=T_{jj}/\left|T_{jj}\right|$.
\item Compute $H_0 = W \log(D) W^{-1}$.
\item Output: $H = \frac{1}{2}H_0^* + \frac{1}{2}H_0$.
\end{enumerate}
\end{alg}

This fails badly in special cases, even for small matrices. Given
\[
U=\left[\begin{array}{cc}
-1 & 0\\
0 & -1
\end{array}\right]
\]
we might consider the approximate diagonalization
\[
W=\left[\begin{array}{cc}
1 & \alpha\\
0 & 1
\end{array}\right],\quad D=\left[\begin{array}{cc}
-1 & 0\\
0 & \exp(-\pi i+\theta i)
\end{array}\right]
\]
 to be rather good if $\theta>0$ is small, even if $\alpha>0$ is
rather large. Indeed
\begin{align*}
WDW^{-1}-U & =\left[\begin{array}{cc}
-1 & \alpha(1+\exp(-\pi i+\theta i))\\
0 & \exp(-\pi i+\theta i)
\end{array}\right]-U\\
 & =(1+\exp(-\pi i+\theta i))\left[\begin{array}{cc}
0 & \alpha\\
0 & 1
\end{array}\right]
\end{align*}
which has norm 
\[
\left|1+\exp(-\pi i+\theta i)\right|\sqrt{1+\alpha^{2}}.
\]
However, if we set
\begin{align*}
K_{0} & =-iW\log(D) W^{-1}\\
 & =\left[\begin{array}{cc}
1 & \alpha\\
0 & 1
\end{array}\right]\left[\begin{array}{cc}
\pi & 0\\
0 & -\pi+\theta
\end{array}\right]\left[\begin{array}{cc}
1 & -\alpha\\
0 & 1
\end{array}\right]\\
 & =\left[\begin{array}{cc}
\pi & -\pi\alpha+\theta\alpha\\
0 & -\pi+\theta
\end{array}\right]\left[\begin{array}{cc}
1 & -\alpha\\
0 & 1
\end{array}\right]\\
 & =\left[\begin{array}{cc}
\pi & -2\pi\alpha+\theta\alpha\\
0 & -\pi+\theta
\end{array}\right]
\end{align*}
we note that this is not close to the actual logarithm, which is diagonal.  It
is not close to any branch of logarithm applied to $U$.
When we enforce the symmetry by $K=\tfrac{1}{2}K_{0}^{*}+\tfrac{1}{2}K_{0}$,
we obtain
\[
K=\left[\begin{array}{cc}
\pi & -\pi\alpha+\frac{1}{2}\theta\alpha\\
-\pi\alpha+\frac{1}{2}\theta\alpha & -\pi+\theta
\end{array}\right].
\]
 When $\alpha=1$ we have
\[
\lim_{\theta\rightarrow0}\left\Vert WDW^{-1}-U\right\Vert =0
\]
and yet
\[
\left[\begin{array}{cc}
\pi & -\pi\alpha+\frac{1}{2}\theta\\
-\pi\alpha+\frac{1}{2}\theta & -\pi+\theta
\end{array}\right]\rightarrow K_{1}
\]
where
\[
K_{1}=\left[\begin{array}{cc}
\pi & -\pi\\
-\pi & -\pi
\end{array}\right].
\]
Working numerically, we find
\[
\lim_{\theta\rightarrow0}\left\Vert e^{iK}-U\right\Vert 
=
\left\Vert e^{iK_{1}}-U\right\Vert 
\approx1.2114
\]
which is dramatically off.

Examining the methods for computing matrix exponentials, both good and bad,
in \cite{moler1978nineteen}, we notice methods based on the Schur decomposition.
In theory, the Schur decomposition of a normal matrix will result in a
an upper triangular factor that is actually diagonal.  So almost normal should
lead to an almost diagonal factor, but this is too naive an approach.  However
we are not concerned with general almost normal matrices just now, only the 
nice special case of almost unitary matrices.  A naive approach using the
Schur decomposition will work well in this specialized situation.

The \texttt{logm} function in MATLAB is optimized to so that 
$H=\mathrm{logm}(U)$
leads to $\mathrm{expm}(H)$ being very close to $U$, even when $U$
is badly conditioned. It is not
optimized to convert the approximate relation $U^{*}U\approx I$ into
the approximate relation $H^{*}\approx-H$. The Schur-Parlett algorithm
\cite{DavHighamMatrixFunctions2003} behind \texttt{logm} and 
\texttt{funm} is primarily intended for entire functions. It can behaive
unexpectedly when eigenvalues are at and near branch points.

\begin{alg} \ \label{alg:logm}
\begin{enumerate}
\item Compute logarithm $H_0$ of $U$ via  \texttt{logm}.
\item Output: $H = \frac{1}{2}H_0^* + \frac{1}{2}H_0$.
\end{enumerate}
\end{alg}

A far more ambitious project would be to find algorithms for
matrix functions that, given almost normal matrices as input,
lead to almost normal matrices on output.  The subject of almost
normal matrices is discussed in a recent
survey \cite{davidson2010essentially} and is full of
subtle problems.

\section{Schur factorization of near unitary matrices}

We define the \emph{deviation from unitary} to be the number
$\left\Vert U^{*}U-I\right\Vert $.

In finite precision arithmetic, we can expect the deviation 
from unitary to almost never equal zero.  Not surprisingly
this error is easy to handle.

\begin{lem}
Suppose $U$ is in $\mathbf{M}_{n}(\mathbb{C})$ and 
$U\mathbf{v}=\lambda\mathbf{v}$
for some unit vector $\mathbf{v}$ and scalar $\lambda$. Then
\[
\left|\left|\lambda\right|^{2}-1\right|
\leq
\left\Vert U^{*}U-I\right\Vert .
\]
\end{lem}

\begin{proof}
Let $P$ denote the orthogonal projection on the one-dimensional
space spanned by $\lambda$. Then
\begin{align*}
\left|\left|\lambda\right|^{2}-1\right| 
& =\left\Vert \left(\left|\lambda\right|^{2}-1\right)P\right\Vert \\
& =\left\Vert P\left(U^{*}U-I\right)P\right\Vert \\
& \leq\left\Vert U^{*}U-I\right\Vert .
\end{align*}
\end{proof}

\begin{rem}
Notice then
\[
\sqrt{1-\left\Vert U^{*}U-I\right\Vert }
\leq
\left|\lambda\right|
\leq
\sqrt{1+\left\Vert U^{*}U-I\right\Vert }
\]
so that as long as $\left\Vert U^{*}U-I\right\Vert <1$ we have 
\[
\left|\left|\lambda\right|-1\right|
\leq
\left\Vert U^{*}U-I\right\Vert .
\]
\end{rem}

\begin{lem}
\label{lem:TriangFrobeneous} Suppose $T$ in $\mathbf{M}_{n}(\mathbb{C})$
is upper triangular and let $D$ denote the diagonal matrix corresponding
to the diagonal of $T$. Then
\[
\left\Vert T-D\right\Vert _{\mathrm{F}}
\leq
\sqrt{2(n-1)}\left\Vert T^{*}T-I\right\Vert ^{\frac{1}{2}}.
\]
\end{lem}

\begin{proof}
We know that the $T_{jj}$ are all eigenvalues of $T$ so
\[
\left|\left|T_{jj}\right|^{2}-1\right|
\leq
\left\Vert T^{*}T-I\right\Vert .
\]
Every element on the diagonal of a matrix has absolute value at most
the norm of that matrix. Applied to $T^{*}T-I$ this tells us
\[
\left\Vert T^{*}T-I\right\Vert 
\geq
\left|-1+\sum_{k=1}^{j}\left|T_{kj}\right|^{2}\right|
\geq
\sum_{k=1}^{j-1}\left|T_{kj}\right|^{2}-\left|\left|T_{jj}\right|^{2}-1\right|
\]
so
\[
\sum_{k=1}^{j-1}\left|T_{kj}\right|^{2}\leq2\left\Vert T^{*}T-I\right\Vert .
\]
Summing these error bounds we learn
\[
\left(\left\Vert T-D\right\Vert _{\mathrm{F}}\right)^{2}
\leq
2(n-1)\left\Vert T^{*}T-I\right\Vert .
\]
\end{proof}

\begin{rem} \label{estimateInverseEigen}
We get here the estimate
\[
\left\Vert T-D\right\Vert _{\mathrm{F}}
\leq
\sqrt{2(n-1)}\left\Vert T^{*}T-I\right\Vert _{\mathrm{F}}^{\frac{1}{2}}
\]
which we can compare to Henrici's estimate \cite{HenriciNonNormal1962}
\[
\left\Vert T-D\right\Vert _{\mathrm{F}}
\leq
\left(\frac{n^{3}-n}{12}\right)^{\frac{1}{2}}
\left\Vert T^{*}T-TT^{*}\right\Vert _{\mathrm{F}}^{\frac{1}{2}}.
\]
As expected, we are getting a better estimate for almost unitary matrices
than works in the more general case of almost normal matrices. 
\end{rem}

\begin{lem}
Suppose $T$ in $\mathbf{M}_{n}(\mathbb{C})$ is upper triangular
with $\left\Vert T^{*}T-I\right\Vert <1$. If we let $D$ denote the
diagonal matrix with
\[
D_{jj}=\frac{1}{\left|T_{jj}\right|}T_{jj}
\]
then
\[
\left\Vert T-D\right\Vert 
\leq
\left(\sqrt{2(n-1)}+1\right)\left\Vert T^{*}T-I\right\Vert ^{\frac{1}{2}}.
\]
\end{lem}
\begin{proof}
The previous lemma shows that with $E$ diagonal and $E_{jj}=T_{jj}$
we have
\[
\left\Vert T-E\right\Vert \leq\left\Vert T-E\right\Vert _{\mathrm{F}}
\leq
\sqrt{2(n-1)}\left\Vert T^{*}T-I\right\Vert ^{\frac{1}{2}}.
\]
Dealing with diagonal matrices is easy, and we find
\[
\| D - E \| 
= \sup_j \left| T_{jj} - \frac{1}{\left|T_{jj}\right|}T_{jj} \right|
\leq  \| T^*T - I \|
\]
by Remark~\ref{estimateInverseEigen}
so
\begin{align*}
\left\Vert T-D\right\Vert  
 & \leq\sqrt{2(n-1)}\left\Vert T^{*}T-I\right\Vert ^{\frac{1}{2}}+\left\Vert T^{*}T-I\right\Vert \\
 & \leq\left(\sqrt{2(n-1)}+1\right)\left\Vert T^{*}T-I\right\Vert ^{\frac{1}{2}}.
\end{align*}
\end{proof}

\begin{thm}
\label{thm:SchurNearNormal} 
Suppose $U$ is in $\mathbf{M}_{n}(\mathbb{C})$
and that unitary $Q$ and upper-triangular matrix $T$ are a Schur
factorization for $U$, meaning $U=QTQ^{*}$. If we define $D$ to
be the diagonal unitary matrix with 
\[
D_{jj}=\frac{1}{\left|T_{jj}\right|}T_{jj}
\]
then
\[
\left\Vert U-QDQ^{*}\right\Vert 
\leq
\left(\sqrt{2(n-1)}+1\right)\left\Vert U^{*}U-I\right\Vert ^{\frac{1}{2}}.
\]
\end{thm}

\begin{proof}
We have
\[
\left\Vert T^{*}T-I\right\Vert
=
\left\Vert Q^{*}\left(U^{*}U-I\right)Q\right\Vert
=
\left\Vert U^{*}U-I\right\Vert 
\]
and
\[
\left\Vert T-D\right\Vert 
=
\left\Vert Q^{*}\left(U-QDQ^{*}\right)Q\right\Vert 
=
\left\Vert U-QDQ^{*}\right\Vert 
\]
so this follows immediately.
\end{proof}

\section{Diagonalizing matrices that are close to unitary}

We may face a matrix $U$ where $U^{*}U$ is as close to $I$ as can
be expected within the constraints of machine precision. For Hermitian
matrices, the eigensolvers in LAPACK  produce eigenvectors that
are ``always nearly orthogonal to working precision'' 
\cite[\S 4.7.1]{anderson1999lapack}.  No comparable promise is made 
in other eigensolvers, but off-the-shelf
algorithm such as ZGEES in LAPACK computing a Schur factorization form
a good substitute.  We get a simple algorithm for finding a unitary
eigensolver for a unitary matrix.

\begin{alg} \ \label{alg:Schur}
\begin{enumerate}
\item Compute a Schur factorization, $Q$ unitary and $T$ upper-triangular,
$QTQ^{*} \approx U$.
\item Create a unitary diagonal matrix $D$ via 
$D_{jj}=T_{jj}/\left|T_{jj}\right|$.
\item Compute $Q \log(D) Q^*$.
\end{enumerate}
\end{alg}

If $\left\Vert U^{*}U-I\right\Vert $ is a larger than machine
precision, then we can compute the unitary part of $U$ and
proceed as above.  For simplicity of programming, we will compute
the polar decomposition via the singular value decomposition.  
We are not advocating this method, heeding the warnings in
\cite{higham1994parallel}.

\begin{alg} \ \label{alg:polarSchur}
\begin{enumerate}
\item Set $V = U\left(U^*U\right)^{-\frac{1}{2}}$.
\item Compute a Schur factorization, $Q$ unitary and $T$ upper-triangular,
$QTQ^{*} \approx V$.
\item Create a unitary diagonal matrix $D$ via $ $$D_{jj}=T_{jj}/\left|T_{jj}\right|$.
\item Compute $Q \log(D) C^*$;
\end{enumerate}
\end{alg}

Our focus is on $U$ with deviation from unitary in the range
$(0,0.1)$.  
Newton's method of approximating the unitary part of $U$ is
very effective in this situation.  Newton's method here sets
$V=U$ and iterates the replacement
$V\leftarrow\tfrac{1}{2}\left(V+\left(V^{-1}\right)^{*}\right)$.
For our purposes, it makes sense to use two interations.

\begin{alg} \ \label{alg:NewtonSchur}
\begin{enumerate}
\item Set $V_1 = \tfrac{1}{2}\left(U+\left(U^{-1}\right)^{*}\right)$.
\item Set $V = \tfrac{1}{2}\left(V_1+\left(V_1^{-1}\right)^{*}\right)$.
\item Compute a Schur factorization, $Q$ unitary and $T$ upper-triangular,
$QTQ^{*} \approx V$.
\item Create a unitary diagonal matrix $D$ via $ $$D_{jj}=T_{jj}/\left|T_{jj}\right|$.
\item Compute $Q \log(D) Q^*$.
\end{enumerate}
\end{alg}

One could compute when to stop the iterations for best
accuracy, following,
\cite{higham1994parallel}.  As simpler methods work here, where
we have such well-conditioned matrices, we do not pursue
this option.

\begin{table}

\begin{tabular}{|c|>{\centering}p{1in}|>{\centering}p{1in}|c|c|c|c|}
\hline 
 & deviation from unitary & \multicolumn{5}{c|}{Backwards Error $\left\Vert e^{-iH}-U\right\Vert _{\strut}^{\strut}$}\tabularnewline
\hline 
$n$ & $\left\Vert U^{*}U-I\right\Vert _{\strut}^{\strut}$ & Algorithm 1 & Algorithm 2 & Algorithm 3 & Algorithm 4 & Algorithm 5\tabularnewline
\hline 
\hline 
8 & 4.11082e-15 & 0.12285 & 0.12340 & 4.30902e-15 & 4.41189e-15 & 4.13976e-15\tabularnewline
\hline 
16 & 5.02961e-15 & 0.04407 & 0.04465 & 6.31606e-15 & 6.47906e-15 & 6.13171e-15\tabularnewline
\hline 
32 & 6.33082e-15 & 0.09286 & 0.09099 & 9.36391e-15 & 9.30067e-15 & 8.99073e-15\tabularnewline
\hline 
64 & 1.10432e-14 & 0.01952 & 0.01641 & 1.38378e-14 & 1.39124e-14 & 1.32675e-14\tabularnewline
\hline 
128 & 1.34734e-14 & 0.01999 & 0.02239 & 2.29980e-14 & 2.32533e-14 & 2.26790e-14\tabularnewline
\hline 
256 & 3.19324e-14 & 0.06131 & 0.06158 & 4.49885e-14 & 4.31729e-14 & 4.42639e-14\tabularnewline
\hline 
\end{tabular}

\caption{We compare algorithms that produce matrices $H$ with $H^{*}=H$ exactly
but with $e^{iH}\approx U$ with varying accuracy. These are run on
unitary matrices with at least two eigenvalues near $-1$. For each
$n$ the data shown are averaged over 30 test matrices.
The noise variable in the code
in the Appendix was set to $10^{-15}n^{-0.56}$.
The exponent $-0.56$ was determined to keep the deviation from 
unitary close to constant across different matrix sizes.
\label{tab:smallNoiseErr}
}

\end{table}

Even one iteration of Newton's method has advantages.

\begin{lem}
Suppose $U$ is in $\mathbf{M}_{n}(\mathbb{C})$ and 
$\left\Vert U^{*}U-I\right\Vert \leq\frac{3}{4}$.
If 
\[
V=\frac{1}{2}\left(U+\left(U^{-1}\right)^{*}\right)
\]
then
\[
\left\Vert V^{*}V-I\right\Vert \leq\left\Vert U^{*}U-I\right\Vert ^{2}
\]
and
\[
\left\Vert U-V\right\Vert \leq\left\Vert U^{*}U-I\right\Vert .
\]
\end{lem}

\begin{proof}
Let $U=D\Omega$ be the alternate-side polar decomposition of $U$,
so $\Omega$ is unitary and $D$ is positive semi-definite. Then
\[
\left\Vert U^{*}U-I\right\Vert 
=
\left\Vert \Omega^{*}D^{2}\Omega-I\right\Vert 
=
\left\Vert D^{2}-I\right\Vert =\sup_{\lambda\in\sigma(D)}\left|\lambda^{2}-1\right|
\]
and 
\[
\left\Vert V^{*}V-I\right\Vert 
=
\left\Vert \left|\frac{1}{2}\left(D+D^{-1}\right)\right|^{2}-I\right\Vert 
=
\sup_{\lambda\in\sigma(D)}\left|\tfrac{1}{4}\left(\lambda+\lambda^{-1}\right)^{2}-1\right|
\]
 and
\[
\left\Vert U-V\right\Vert 
=
\left\Vert D-\frac{1}{2}\left(D+D^{-1}\right)\right\Vert 
=
\sup_{\lambda\in\sigma(D)}\left|\lambda-\tfrac{1}{2}\left(\lambda+\lambda^{-1}\right)\right|.
\]
This lemma reduces to routine algebra, showing that 
\[
\frac{1}{2}\leq\lambda\leq\frac{\sqrt{7}}{2}
\]
 implies
\[
\left|\tfrac{1}{4}\left(\lambda+\lambda^{-1}\right)^{2}-1\right|\leq\left|\lambda^{2}-1\right|^{2}
\]
 and
\[
\left|\lambda-\tfrac{1}{2}\left(\lambda+\lambda^{-1}\right)\right|
\leq
\left|\lambda^{2}-1\right|.
\]
\end{proof}

\begin{table}
\vspace{0.5cm}
\begin{tabular}{|c|>{\centering}p{1in}|>{\centering}p{1in}|c|c|c|c|}
\hline 
 & deviation from unitary & \multicolumn{5}{c|}{Time }\tabularnewline
\hline 
$n$ & $\left\Vert U^{*}U-I\right\Vert _{\strut}^{\strut}$ & Algorithm 1 & Algorithm 2 & Algorithm 3 & Algorithm 4 & Algorithm 5\tabularnewline
\hline 
\hline 
8 & 4.11082e-15 & 0.00013s & 0.00190s & 0.00009s & 0.00014s & 0.00016s\tabularnewline
\hline 
16 & 5.02961e-15 & 0.00034s & 0.00618s & 0.00024s & 0.00036s & 0.00037s\tabularnewline
\hline 
32 & 6.33082e-15 & 0.00139s & 0.02104s & 0.00104s & 0.00146s & 0.00145s\tabularnewline
\hline 
64 & 1.10432e-14 & 0.00877s & 0.05853s & 0.00719s & 0.00926s & 0.00886s\tabularnewline
\hline 
128 & 1.34734e-14 & 0.06111s & 0.13750s & 0.05442s & 0.06843s & 0.06218s\tabularnewline
\hline 
256 & 3.19324e-14 & 0.33397s & 0.73867s & 0.29425s & 0.35710s & 0.32437s\tabularnewline
\hline 
\end{tabular}\caption{Timing data for the algorithms and data set from
Table~\ref{tab:smallNoiseErr}. Times were computed in a non-parallel environment.
\label{tab:time}
}
\end{table}

\begin{thm}
Suppose $U$ is in $\mathbf{M}_{n}(\mathbb{C})$ with 
$\left\Vert U^{*}U-I\right\Vert \leq\frac{3}{4}$
and let 
\[
V=\frac{1}{2}\left(U+\left(U^{-1}\right)^{*}\right).
\]
If there is a unitary $Q$ and an upper-triangular matrix $T$ so
that $V=QTQ^{*}$, and if we define $D$ to be the diagonal unitary
matrix with 
\[
D_{jj}=\frac{1}{\left|T_{jj}\right|}T_{jj}
\]
then
\[
\left\Vert U-QDQ^{*}\right\Vert 
\leq
\left(\sqrt{2(n-1)}+2\right)\left\Vert U^{*}U-I\right\Vert .
\]
\end{thm}

\begin{proof}
We have 
\[
\left\Vert V-QDQ^{*}\right\Vert 
\leq
\left(\sqrt{2(n-1)}+1\right)\left\Vert V^{*}V-I\right\Vert ^{\frac{1}{2}}
\leq
\left(\sqrt{2(n-1)}+1\right)\left\Vert U^{*}U-I\right\Vert 
\]
and
\[
\left\Vert U-V\right\Vert \leq\left\Vert U^{*}U-I\right\Vert .
\]
\end{proof}

\begin{table}
\begin{tabular}{|c|>{\centering}p{1in}|>{\centering}p{1in}|c|c|c|c|}
\hline 
 & deviation from unitary & \multicolumn{5}{c|}{Backwards Error $\left\Vert e^{-iH}-U\right\Vert _{\strut}^{\strut}$}\tabularnewline
\hline 
$n$ & $\left\Vert U^{*}U-I\right\Vert _{\strut}^{\strut}$ & Algorithm 1 & Algorithm 2 & Algorithm 3 & Algorithm 4 & Algorithm 5\tabularnewline
\hline 
\hline 
8 & 1.19608e-05 & 0.33294 & 0.33294 & 9.42901e-06 & 5.98042e-06 & 5.98042e-06\tabularnewline
\hline 
16 & 1.27157e-05 & 0.18364 & 0.18364 & 1.01912e-05 & 6.35784e-06 & 6.35784e-06\tabularnewline
\hline 
32 & 1.25555e-05 & 0.24113 & 0.24113 & 1.01824e-05 & 6.27775e-06 & 6.27775e-06\tabularnewline
\hline 
64 & 1.21902e-05 & 0.21304 & 0.21304 & 9.96141e-06 & 6.09511e-06 & 6.09511e-06\tabularnewline
\hline 
128 & 1.18934e-05 & 0.24809 & 0.24809 & 9.74603e-06 & 5.94671e-06 & 5.94671e-06\tabularnewline
\hline 
256 & 1.15296e-05 & 0.25532 & 0.25532 & 9.45829e-06 & 5.76481e-06 & 5.76481e-06\tabularnewline
\hline 
\end{tabular}

\caption{As in Table~\ref{tab:smallNoiseErr}, but
with the noise variable set to $10^{-5}n^{-0.56}$. 
\label{tab:medNoiseErr}
}
\end{table}

It is worth keeping in mind that the spectral decomposition of the
polar part of $U$ leads to the theoretical best unitary diagonalization,
with $D$ diagonal, $Q$ unitary and
\[
\left\Vert U-Q^{*}DQ\right\Vert 
=
\max\left\{ |\lambda-1|\ \left|\strut\lambda\in\sigma\left(\left(U^{*}U\right)^{\frac{1}{2}}\right)\right.\right\} 
\]
and the best general lower bound is 
\[
\left\Vert U-Q^{*}DQ\right\Vert \geq\frac{1}{2}\left\Vert U^{*}U-I\right\Vert .
\]

\begin{lem}
Suppose $U$ is in $\mathbf{M}_{n}(\mathbb{C})$ and 
$\left\Vert U^{*}U-I\right\Vert \leq\frac{3}{4}$.
If 
\[
V_{1}=\frac{1}{2}\left(U+\left(U^{-1}\right)^{*}\right)
\]
and
\[
V=\frac{1}{2}\left(V_{1}+\left(V_{1}^{-1}\right)^{*}\right)
\]
then
\[
\left\Vert V^{*}V-I\right\Vert \leq\tfrac{4}{25}\left\Vert U^{*}U-I\right\Vert ^{4}
\]
 and
\[
\left\Vert U-V\right\Vert \leq\tfrac{7}{10}\left\Vert U^{*}U-I\right\Vert .
\]
\end{lem}

\begin{proof}
Tracking the spectrum of positive parts as before, we are looking
at 
\[
\tfrac{1}{2}\frac{\left(\tfrac{1}{2}\frac{\lambda^{2}+1}{\lambda}\right)^{2}+1}{\left(\tfrac{1}{2}\frac{\lambda^{2}+1}{\lambda}\right)}
=
\frac{\lambda^{4}+6\lambda^{2}+1}{4\lambda\left(\lambda^{2}+1\right)}
\]
and all that is needed is some algebra and calculus to verify that
\[
\frac{1}{2}\leq\lambda\leq\frac{\sqrt{7}}{2}
\]
implies
\[
\left|\left(\frac{\lambda^{4}+6\lambda^{2}+1}{4\lambda\left(\lambda^{2}+1\right)}\right)^{2}-1\right|
\leq
\frac{4}{25}\left|\lambda^{2}-1\right|^{4}
\]
 and
\[
\left|\frac{\lambda^{4}+6\lambda^{2}+1}{4\lambda\left(\lambda^{2}+1\right)}-\lambda\right|
\leq
\frac{7}{10}\left|\lambda^{2}-1\right|.
\]
\end{proof}

\begin{table}
\begin{tabular}{|c|>{\centering}p{1in}|>{\centering}p{1in}|c|c|c|c|}
\hline 
 & deviation from unitary & \multicolumn{5}{c|}{Backwards Error $\left\Vert e^{-iH}-U\right\Vert _{\strut}^{\strut}$}\tabularnewline
\hline 
$n$ & $\left\Vert U^{*}U-I\right\Vert _{\strut}^{\strut}$ & Algorithm 1 & Algorithm 2 & Algorithm 3 & Algorithm 4 & Algorithm 5\tabularnewline
\hline 
\hline 
8 & 3.98986e-01 & 0.48072 & 0.48083 & 2.91776e-01 & 1.87829e-01 & 1.87829e-01\tabularnewline
\hline 
16 & 3.99794e-01 & 0.76608 & 0.76629 & 3.00304e-01 & 1.88836e-01 & 1.88836e-01\tabularnewline
\hline 
32 & 4.01659e-01 & 0.74454 & 0.74577 & 3.02891e-01 & 1.86026e-01 & 1.86026e-01\tabularnewline
\hline 
64 & 4.04556e-01 & 0.81147 & 0.81217 & 2.99941e-01 & 1.85677e-01 & 1.85677e-01\tabularnewline
\hline 
128 & 3.92340e-01 & 0.86934 & 0.87020 & 2.90461e-01 & 1.79970e-01 & 1.79970e-01\tabularnewline
\hline 
256 & 3.81114e-01 & 1.23535 & 1.23575 & 2.84104e-01 & 1.75207e-01 & 1.75207e-01\tabularnewline
\hline 
\end{tabular}

\caption{As in Table~\ref{tab:smallNoiseErr}, but with
the noise variable set to $0.3n^{-0.56}$. 
\label{tab:largeNoiseErr}
}
\end{table}

As before, we use this to get an estimate on the algorithm that uses
two iterations of Newton's method followed by a Schur decomposition. 

\begin{thm}
Suppose $U$ is in $\mathbf{M}_{n}(\mathbb{C})$ with 
$\left\Vert U^{*}U-I\right\Vert \leq\frac{3}{4}$
and $n\geq 3$.  Let
\[
V_{1}=\frac{1}{2}\left(U+\left(U^{-1}\right)^{*}\right)
\]
and
\[
V=\frac{1}{2}\left(V_{1}+\left(V_{1}^{-1}\right)^{*}\right)
\]
If there is a unitary $Q$ and an upper-triangular matrix $T$ so
that $V=QTQ^{*}$, and if we define $D$ to be the diagonal unitary
matrix with 
\[
D_{jj}=\frac{1}{\left|T_{jj}\right|}T_{jj}
\]
then
\[
\left\Vert U-QDQ^{*}\right\Vert 
\leq
\tfrac{7}{10}\sqrt{n}\left\Vert U^{*}U-I\right\Vert ^{2}+\tfrac{7}{10}\left\Vert U^{*}U-I\right\Vert .
\]
\end{thm}

\begin{proof}
We have
\[
\left\Vert V^{*}V-I\right\Vert \leq\frac{4}{25}\left\Vert U^{*}U-I\right\Vert ^{4}
\]
and
\begin{align*}
\left\Vert V-QDQ^{*}\right\Vert  
 & \leq\left(\sqrt{2(n-1)}+1\right)\left\Vert V^{*}V-I\right\Vert ^{\frac{1}{2}}\\
 & \leq\left(\sqrt{2(n-1)}+1\right)\left(\frac{4}{25}\left\Vert U^{*}U-I\right\Vert ^{4}\right)^{\frac{1}{2}}\\
 & \leq\frac{7}{10}\sqrt{n}\left(\left\Vert U^{*}U-I\right\Vert ^{4}\right)^{\frac{1}{2}}
\end{align*}
and
\[
\left\Vert U-V\right\Vert \leq\tfrac{7}{10}\left\Vert U^{*}U-I\right\Vert .
\]

\end{proof}

Rather extreme input data is need to highlight the advantage of
Algorithm~\ref{alg:NewtonSchur}.  We test algorithms 
\ref{alg:logm}-\ref{alg:NewtonSchur}
on unitaries, and approximate
unitaries, that have multiple eigenvalues very near $-1$. See
Tables \ref{tab:smallNoiseErr}-\ref{tab:largeNoiseErr}.

Not much can be said regarding the error in the
output $H$ of Algorithm~\ref{alg:NewtonSchur} vs.\  the 
``true logarithm.'' The trouble is that the set of ``exactly connect''
answers jumps around in a most discontinious fashion.
If $\alpha,\beta$ are in $(-\pi,\pi]$ with $\alpha\approx-\pi$
and $\beta\approx\pi$ then for any unitary $Q$, consider the unitary
\[
U=Q\left[\begin{array}{cc}
e^{i\alpha} & 0\\
0 & e^{i\beta}
\end{array}\right]Q^{*}.
\]
The only exactly correct output given $U$ on input is
\[
Q\left[\begin{array}{cc}
\alpha & 0\\
0 & \beta
\end{array}\right]Q^{*}
\]
and this greatly depends on $Q$.

\section{$J$-skew-symmetric unitaries \label{sec:Self-dual-unitaries}}

When dealing with Hamiltonians for systems with certain time-reversal
symmetry, we need an extra involution on matrices, the\emph{ dual}.
Working in $N$-by-$N$ blocks, the dual operation on  $n$-by-$n$ matrices
with $n=2N$ is defined as
\[
\left[\begin{array}{cc}
A & B\\
C & D
\end{array}\right]^{\sharp}=\left[\begin{array}{cc}
D^{\mathrm{T}} & -B^{\mathrm{T}}\\
-C^{\mathrm{T}} & A^{\mathrm{T}}
\end{array}\right].
\]
It is easy to check this obeys the same axiom as the transpose. In
particular it commutes with the adjoint. See \cite{LoringQuaternions},
for example. The dual operation is not unique, but is fixed once we
specify 
\[
J=\left[\begin{array}{cc}
0 & I\\
-I & 0
\end{array}\right].
\]
Then $X^{\sharp}=-JX^{\mathrm{T}}J$. A matrix $X$ is $J$-skew-symmetric
when $X^{\sharp}=X$ and so if and only if $\left(XJ\right)^{\mathrm{T}}=-XJ$.

\begin{table}
\begin{tabular}{|c|>{\centering}p{1in}|>{\centering}p{1in}|c|c|}
\hline 
 & deviation from unitary & \multicolumn{3}{c|}{Backwards Error $\left\Vert e^{-iH}-U\right\Vert _{\strut}^{\strut}$}\tabularnewline
\hline 
$n$ & $\left\Vert U^{*}U-I\right\Vert _{\strut}^{\strut}$ & Algorithm 1A & Algorithm 2A & Algorithm 6\tabularnewline
\hline 
\hline 
8 & 2.96904e-15 & 1.17188 & 1.12520 & 3.27683e-15\tabularnewline
\hline 
16 & 3.27269e-15 & 1.10402 & 1.11826 & 4.50363e-15\tabularnewline
\hline 
32 & 4.12604e-15 & 1.13947 & 1.15103 & 6.68904e-15\tabularnewline
\hline 
64 & 8.33263e-15 & 0.77344 & 0.77214 & 1.00208e-14\tabularnewline
\hline 
128 & 1.02179e-14 & 1.46559 & 1.46102 & 1.52540e-14\tabularnewline
\hline 
256 & 2.46375e-14 & 1.19984 & 1.19987 & 2.78177e-14\tabularnewline
\hline 
\end{tabular}

\caption{We compare algorithms that produce matrices $H$ with $H^{*}=H^{\sharp}=H$
exactly but with $e^{iH}\approx U$ with varying accuracy. These are
run on unitary, $J$-skew-symmetric matrices with at least four eigenvalues
near $-1$. For each $n$ the results
shown are averaged iver $30$ test matrices. The noise
variable was set to $10^{-15}n^{-0.56}$ in the code in the appendix.
\label{tab:selfDualsmall}
}

\end{table}

There is a \emph{$J$-skew-symmetric Schur decomposition} result for $J$-skew-symmetric
matrices, a.k.a.~the skew-Hamiltonian Schur decomposition result
from \cite{benner-skew}.

\begin{thm} \label{thm:self-dual-algorithm}
Let $n=2N$.
If $X^{\sharp}=X$ in $\mathbf{M}_{n}(\mathbb{C})$ then there is
a unitary $Q$ with $Q^{\sharp}=Q^{*}$ and a matrix $S$ with $X=QSQ^{*}$
and so that 
\[
S=\left[\begin{array}{cc}
T & B\\
0 & T^{\mathrm{T}}
\end{array}\right]
\]
where $T$ is upper triangular. Moreover there is an $O(n^{3})$ algorithm
to find $Q$ and $T$.
\end{thm}

\begin{proof}
This is the complex analog of \cite[\S 2.1]{benner-skew}, discussed
in detail in \cite[\S 9.1]{HastLorTheoryPractice}. One uses some
variation on the Paige / Van Loan algorithm \cite[\S 2.1]{benner-skew},
involving careful combinations of Givens rotations and partial Householder
reflections to create a unitary $Q_{1}$ with $Q_{1}^{\sharp}=Q_{1}^{*}$
and so that $X=Q_{1}S_{1}Q_{1}^{*}$ with 
\[
S_{1}=\left[\begin{array}{cc}
T_{1} & B_{1}\\
0 & T_{1}^{\mathrm{T}}
\end{array}\right].
\]
 An ordinary Schur decomposition $T_{1}=WTW^{*}$ now finishes the
job, as
\[
Q=Q_{1}\left[\begin{array}{cc}
W\\
 & \overline{W}
\end{array}\right]
\]
has the needed symmetry, while
\[
S=\left[\begin{array}{cc}
W\\
 & \overline{W}
\end{array}\right]^{*}S_{1}\left[\begin{array}{cc}
W\\
 & \overline{W}
\end{array}\right]
\]
has the correct block structure. 
\end{proof}

For comparision purposes, we say we are using Algorithm~\ref{alg:logm}A, etc,
if we add a final step to the algorithms above that replaces 
$H$ by $\tfrac{1}{2}H^\sharp.$

\begin{table}
\vspace{0.5cm}
\begin{tabular}{|c|>{\centering}p{1in}|>{\centering}p{1in}|c|c|}
\hline 
 & deviation from unitary & \multicolumn{3}{c|}{Time }\tabularnewline
\hline 
$n$ & $\left\Vert U^{*}U-I\right\Vert _{\strut}^{\strut}$ & Algorithm 1A & Algorithm 2A & Algorithm 6\tabularnewline
\hline 
\hline 
8 & 2.96904e-15 & 0.00019s & 0.00223s & 0.00250s\tabularnewline
\hline 
16 & 3.27269e-15 & 0.00034s & 0.00503s & 0.00616s\tabularnewline
\hline 
32 & 4.12604e-15 & 0.00124s & 0.01752s & 0.01492s\tabularnewline
\hline 
64 & 8.33263e-15 & 0.00773s & 0.05169s & 0.03858s\tabularnewline
\hline 
128 & 1.02179e-14 & 0.05879s & 0.16380s & 0.14007s\tabularnewline
\hline 
256 & 2.46375e-14 & 0.34309s & 0.52146s & 0.79723s\tabularnewline
\hline 
\end{tabular}\caption{Timing data for the algorithms and data set from
Table~\ref{tab:selfDualsmall}.
\label{tab:timingDual}
}
\end{table}

Notice that $S$ becomes upper-triangular if we just reverse the order
of half the basis elements.

\begin{thm}
\label{thm:SchurNearNormalStructured} 
Let $n=2N$.
Suppose $U$ is a $J$-skew-symmetric
matrix in $\mathbf{M}_{n}(\mathbb{C})$ and that unitary $Q$ and
matrix $T$ form a $J$-skew-symmetric Schur decomposition for $U$. If we define
$D$ to be the diagonal unitary matrix with 
\[
D_{jj}=\frac{1}{\left|T_{jj}\right|}T_{jj}
\]
then $D$ is $J$-skew-symmetric, diagonal and 
\[
\left\Vert U-QDQ^{*}\right\Vert
\leq
\left(\sqrt{2(n-1)}+1\right)\left\Vert U^{*}U-I\right\Vert ^{\frac{1}{2}}.
\]
\end{thm}

If we approximate the polar part of $U$ by Newton's method, the symmetry
$U^{\sharp}=U$ is preserved at each iteration, since
\[
\left(\frac{1}{2}\left(U+\left(U^{-1}\right)^{*}\right)\right)^{\sharp}
=
\frac{1}{2}U^{\sharp}+\left(\left(U^{-1}\right)^{*}\right)^{\sharp}
=
\frac{1}{2}U^{\sharp}+\left(\left(U^{-1}\right)^{\sharp}\right)^{*}
=\frac{1}{2}U^{\sharp}+\left(\left(U^{\sharp}\right)^{-1}\right)^{*}.
\]

\begin{thm}
\label{thm:self_dual_diag} 
Let $n=2N$.
Suppose $U^{\sharp}=U$ is in 
$\mathbf{M}_{n}(\mathbb{C})$
with $\left\Vert U^{*}U-I\right\Vert \leq\frac{3}{4}$ and let 
\[
V_{1}=\frac{1}{2}\left(U+\left(U^{-1}\right)^{*}\right)
\]
and 
\[
V=\frac{1}{2}\left(V_{1}+\left(V_{1}^{-1}\right)^{*}\right)
\]
If a unitary $Q$ and matrix $T$ form a $J$-skew-symmetric Schur decomposition
for $V$, and if we define $D$ the diagonal unitary matrix with 
\[
D_{jj}=\frac{1}{\left|T_{jj}\right|}T_{jj},
\]
 then $D$ is $J$-skew-symmetric and 
\[
\left\Vert U-QDQ^{*}\right\Vert 
\leq
\tfrac{7}{10}\sqrt{n}\left\Vert U^{*}U-I\right\Vert ^{2}+\tfrac{7}{10}\left\Vert U^{*}U-I\right\Vert .
\]
\end{thm}

\begin{table}
\begin{tabular}{|c|>{\centering}p{1in}|>{\centering}p{1in}|c|c|}
\hline 
 & deviation from unitary & \multicolumn{3}{c|}{Backwards Error $\left\Vert e^{-iH}-U\right\Vert _{\strut}^{\strut}$}\tabularnewline
\hline 
$n$ & $\left\Vert U^{*}U-I\right\Vert _{\strut}^{\strut}$ & Algorithm 1A & Algorithm 2A & Algorithm 6\tabularnewline
\hline 
\hline 
8 & 6.80865e-06 & 0.50219 & 0.50219 & 3.40432e-06\tabularnewline
\hline 
16 & 7.91995e-06 & 0.27931 & 0.27931 & 3.95997e-06\tabularnewline
\hline 
32 & 8.21901e-06 & 0.32397 & 0.32397 & 4.10950e-06\tabularnewline
\hline 
64 & 8.30081e-06 & 0.43814 & 0.43814 & 4.15041e-06\tabularnewline
\hline 
128 & 8.15713e-06 & 0.30607 & 0.30607 & 4.07856e-06\tabularnewline
\hline 
256 & 8.02983e-06 & 0.38732 & 0.38732 & 4.01491e-06\tabularnewline
\hline 
\end{tabular}

\caption{As in Table~\ref{tab:selfDualsmall}, but with
the noise variable set to $10^{-5}n^{-0.56}$.
\label{tab:selfDualmed}
}
\end{table}

Theorem~\ref{thm:self_dual_diag} gives theoretical justification for
that the following algorithm produces approximately correct output.
This is very similar to the algorithm used in \cite{LorHastHgTe}. 
Since this is built out of known algorithms, we don't have 
anything to say it is $\mathcal{O}(n^3)$.

\begin{alg} \ \label{alg:selfDual}
\begin{enumerate}
\item Set $V_1 = \tfrac{1}{2}\left(U+\left(U^{-1}\right)^{*}\right)$.
\item Set $V = \tfrac{1}{2}\left(V_1+\left(V_1^{-1}\right)^{*}\right)$.
\item Compute a $J$-skew-symmetric Schur factorization, as in Theorem~\ref{thm:self_dual_diag},
with $QTQ^{*} \approx V$.
\item Create a unitary diagonal matrix $D$ via $D_{jj}=T_{jj}/\left|T_{jj}\right|$.
\item Compute $Q \log(D) Q^*$.
\end{enumerate}
\end{alg}

\begin{table}
\begin{tabular}{|c|>{\centering}p{1in}|>{\centering}p{1in}|c|c|}
\hline 
 & deviation from unitary & \multicolumn{3}{c|}{Backwards Error $\left\Vert e^{-iH}-U\right\Vert _{\strut}^{\strut}$}\tabularnewline
\hline 
$n$ & $\left\Vert U^{*}U-I\right\Vert _{\strut}^{\strut}$ & Algorithm 1A & Algorithm 2A & Algorithm 6\tabularnewline
\hline 
\hline 
8 & 2.12731e-01 & 0.45404 & 0.45282 & 1.05752e-01\tabularnewline
\hline 
16 & 2.43013e-01 & 0.55252 & 0.55205 & 1.18737e-01\tabularnewline
\hline 
32 & 2.51907e-01 & 0.83382 & 0.83331 & 1.21165e-01\tabularnewline
\hline 
64 & 2.61957e-01 & 0.96758 & 0.96743 & 1.24065e-01\tabularnewline
\hline 
128 & 2.62333e-01 & 0.94305 & 0.94325 & 1.23593e-01\tabularnewline
\hline 
256 & 2.56967e-01 & 1.03853 & 1.03868 & 1.21172e-01\tabularnewline
\hline 
\end{tabular}

\caption{As in Table~\ref{tab:selfDualsmall}, but with the noise variable set to $0.3n^{-0.56}$.
\label{tab:selfDualBig}
}
\end{table}

For data on timing and accuracy, see Tables
\ref{tab:selfDualsmall}-\ref{tab:selfDualBig}.
These tables were produced with the MATLAB code listed in the appendix. That
code has almost no optimization. In particular, it does not take advantage
of the symmetries $Q^{\sharp}=Q$ and $D^{\sharp}=D$ that hold at
the top of the loop in the Paige - Van Loan algorithm.

\section{Logarithms in Physics}

In quantum mechanics, it is standard to exponential a skew-Hermitian
operator to get a unitary, as this is how one moves from the Hamiltonian
to the time evolution operator. When then is a periodic time-dependent
Hamiltonian $H_{t}$ of period $T$, the definition of quasi-energy
depends on the Floquet Hamiltonian $H_{F}$ defined via
\[
e{}^{-iH_{F}}=Te^{-i\int_{0}^{T}H_{t}\, dt}.
\]
It does not matter if the principal branch of logarithm is used to
define $H_{F}$ but it is important in numerical studies that $H_{F}$
be computed to be Hermitian. The study of Floquet
topological insulators \cite{lindner2011floquet}
is an important special case of a system with a periodic time-dependent
Hamiltonian.  Some Floquet
topological insulators
can only be explained by keeping track of a form of time-reversal
symmetry \cite[equation 21]{katan2013modulated}.
In that case one has a self-dual Floquet Hamiltonian.

The logarithms of unitary matrices arise in another way
in physics, in particular in the study of the more typical
topological insulators 
where there is a time-independents Hamiltonian.
For finite lattice models of non-interacting fermions,
the dimension and symmetry class determine if
distinct topological phases can occur.  Physically
these phases are ordinary insulators and topological
insulators \cite{fulga2011Wire, teo2010topological}.  

Joint working with Hastings \cite{HastLorTheoryPractice}
established that the ordinary insulating phases can be
characterized by the existence of localized vectors that
form a basis of low-energy space (localized Wannier functions)
where the basis preserves an appropriate symmetry.
We explained in Section 4 of \cite{HastLorTheoryPractice}
how to translate this question into a question about almost
normal, almost unitary, or almost commuting matrices.

The AII symmetry class is one where there is a certain 
flavor of time-reversal invariance.  What this means
mathematically is that one starts with $J$-skew-symmetric and
Hermitian matrices for the Hamiltonian and the position
observables.  There is a promising method for computing
the spin Chern numbers that involves self-dual logarithms.
This was introduced in \cite{LorHastHgTe}.  The formula
in $K$-theory used there is now validated by the
theorems in \cite{LoringQuantKth} and the results in
\cite{LoringQuantKth}.

\section{Acknowledgments}

This work was partially
supported by a grant from the Simons Foundation (208723 to Loring).

\section{Appendix}

\subsection{Code for General Unitaries}

\tt
\begin{verbatim}
function [unitary_error,time,accuracy] = testLogs(n, noise,NumbNewt)
time = 0;
accuracy = 0;

%%%%%%%%%%%%%%%%%%%%%%%%%%%%%%%%%%%%%%%%%%%%%%%%%%%%%%%%%%%
% Create a test matrix, approximately unitary
% Not fully random as we want some spectrum near -1
%%%%%%%%%%%%%%%%%%%%%%%%%%%%%%%%%%%%%%%%%%%%%%%%%%%%%%%%%%%

% exponentiate a random self-adjoint matrix to create a unitary
K = 0.25*( rand(n)+i*rand(n)-rand(n)-i*rand(n));
K = K + K';
K = (4*pi /norm(K))*K;
%exponentiate i*K to get a unitary
Q = expm(i*K);

% create a "random" diagonal, to form the test unitary
D = diag(exp(2*pi*i*[0.5, 0.5, rand(1,n-2)]));
U = Q*D*Q';

%add noise
U = U+noise*(rand(n)+i*rand(n)-rand(n)-i*rand(n));
unitary_error = norm(U'*U - eye(size(U)));

%%%%%%%%%%%%%%%%%%%%%%%%%%%%%%%%%%%%%%%%%%%%%%%%%%%%%%%%%%%
% Get a logarithm by diagonalizing U = W*D*inv(A), then
% force it to be hermitian
%%%%%%%%%%%%%%%%%%%%%%%%%%%%%%%%%%%%%%%%%%%%%%%%%%%%%%%%%%%
tstart = tic;
[W,D] = eig(U);
X = diag(imag(log(diag(D))));
H = W*X*inv(W);
H = (1/2)*(H + H');
	
time(1) = toc(tstart);
accuracy(1) = norm(expm(i*H)-U);

%%%%%%%%%%%%%%%%%%%%%%%%%%%%%%%%%%%%%%%%%%%%%%%%%%%%%%%%%%%
% Get a logarithm using logm, then force it
% to be hermitian
%%%%%%%%%%%%%%%%%%%%%%%%%%%%%%%%%%%%%%%%%%%%%%%%%%%%%%%%%%%
tstart = tic;
H = (-i)*logm(U);
H = (1/2)*(H+H');

time(2) = toc(tstart);
accuracy(2) = norm(expm(i*H)-U);

%%%%%%%%%%%%%%%%%%%%%%%%%%%%%%%%%%%%%%%%%%%%%%%%%%%%%%%%%%%
% Get a logarithm by Schur decomposition then
% force it to be hermitian
%%%%%%%%%%%%%%%%%%%%%%%%%%%%%%%%%%%%%%%%%%%%%%%%%%%%%%%%%%%
tstart = tic;
% compute the Schur factorization
[Q,T] = schur(U);
D = diag(imag(log(diag(T))));
H = Q*D*Q';
H = (1/2)*(H + H');

time(3) = toc(tstart);
accuracy(3) = norm(expm(i*H)-U);

%%%%%%%%%%%%%%%%%%%%%%%%%%%%%%%%%%%%%%%%%%%%%%%%%%%%%%%%%%%
% Get a logarithm from the "exact'' unitary part
% followed by Schur decompositionthen
% force it to be hermitian
%%%%%%%%%%%%%%%%%%%%%%%%%%%%%%%%%%%%%%%%%%%%%%%%%%%%%%%%%%%
tstart = tic;
[A,Z,B] = svd(U);
V = A*B';
% compute the Schur factorization
[Q,T] = schur(V);
D = diag(imag(log(diag(T))));
H = Q*D*Q';
H = (1/2)*(H + H');

time(4) = toc(tstart);
accuracy(4) = norm(expm(i*H)-U);

%%%%%%%%%%%%%%%%%%%%%%%%%%%%%%%%%%%%%%%%%%%%%%%%%%%%%%%%%%%
% Get a logarithm by from an approximate unitary part
% followed by Schur decompositionthen
% force it to be hermitian
%%%%%%%%%%%%%%%%%%%%%%%%%%%%%%%%%%%%%%%%%%%%%%%%%%%%%%%%%%%
tstart = tic;
V = U;
for round = 1:NumbNewt
	V = 0.5*(V + inv(V)');
end
% compute the Schur factorization
[Q,T] = schur(V);
D = diag(imag(log(diag(T))));
H = Q*D*Q';
H = (1/2)*(H + H');

time(5) = toc(tstart);
accuracy(5) = norm(expm(i*H)-U);

end
\end{verbatim}

\subsection{Code for $J$-skew symmetric Unitaries}

\tt
\begin{verbatim}
function [unitary_error,time,accuracy] =testLogsDual(n, noise,NumbNewt)

% force n to be even
n = n + mod(n,2);
N = n/2;

time = 0;
accuracy = 0;

%%%%%%%%%%%%%%%%%%%%%%%%%%%%%%%%%%%%%%%%%%%%%%%%%%%%%%%%%%%
% Create a test matrix, approximately unitary
% Not fully random as we want some spectrum near -1
%%%%%%%%%%%%%%%%%%%%%%%%%%%%%%%%%%%%%%%%%%%%%%%%%%%%%%%%%%%

% exponentiate a random anti-self-adjoint and self-dual matrix to
% create a symplectic unitary
K = 0.25*( rand(2*N)+i*rand(2*N)-rand(2*N)-i*rand(2*N));
K(N+1:2*N,N+1:2*N) = -K(1:N,1:N).';
K = (1/2)*(K - dual(K));
K = (1/2)*(K + K');
K = (4*pi /norm(K))*K;
%exponentiate i*K to get a symplectic unitary
Q = expm(i*K);

% create a "random" diagonal, for form the test unitary
D = exp(2*pi*i*[0.5, 0.5,rand(1,N-2)]);
D = diag([D,D]);
U = Q*D*Q';

%add noise, but keep self-dual
U = U+noise*(rand(2*N)+i*rand(2*N)-rand(2*N)-i*rand(2*N));
U = (1/2)*(U + dual(U));
unitary_error = norm(U'*U - eye(size(U)));

%%%%%%%%%%%%%%%%%%%%%%%%%%%%%%%%%%%%%%%%%%%%%%%%%%%%%%%%%%%
% Get a logarithm by diagonalizing U = W*D*inv(A), then
% force it to be hermitian and self dual
%%%%%%%%%%%%%%%%%%%%%%%%%%%%%%%%%%%%%%%%%%%%%%%%%%%%%%%%%%%
tstart = tic;
[W,D] = eig(U);
X = diag(imag(log(diag(D))));
H = W*X*inv(W);
H = (1/2)*(H + H');
H = (1/2)*(H + dual(H));

time(1) = toc(tstart);
accuracy(1) = norm(expm(i*H)-U);

%%%%%%%%%%%%%%%%%%%%%%%%%%%%%%%%%%%%%%%%%%%%%%%%%%%%%%%%%%%
% Get a logarithm using the default logm, then force it
% to be hermitian and self-dual
%%%%%%%%%%%%%%%%%%%%%%%%%%%%%%%%%%%%%%%%%%%%%%%%%%%%%%%%%%%
tstart = tic;
H = (-i)*logm(U);
H = (1/2)*(H+H');
H = (1/2)*(H + dual(H));

time(2) = toc(tstart);
accuracy(2) = norm(expm(i*H)-U);

%%%%%%%%%%%%%%%%%%%%%%%%%%%%%%%%%%%%%%%%%%%%%%%%%%%%%%%%%%%
% Get a logarithm using Newton's method
% followed by a stuctured Schur decomposition, then force it
% to be hermitian and self-dual
%%%%%%%%%%%%%%%%%%%%%%%%%%%%%%%%%%%%%%%%%%%%%%%%%%%%%%%%%%%
tstart = tic;
%Move toward a unitary with Newton's method
V = U;
for round = 1:NumbNewt
	V = 0.5*(V + inv(V)');
end
% compute the Schur factorization

% First the Paige - Van Loan algorithm
[Q,D] = PVL(V);
D = D(1:N,1:N);
% compute the little Schur factorization, finalize Q and D.
[Q1,T1] = schur(D);
Q(1:2*N,1:N) = Q(1:2*N,1:N)*Q1;
Q(1:2*N,1+N:2*N) = Q(1:2*N,1+N:2*N)*conj(Q1);
D = [diag(imag(log(diag(T1))))];
D = [D, zeros(N); zeros(N), D.'];
H = Q*D*Q';

H = (1/2)*(H+H');
H = (1/2)*(H + dual(H));

time(3) = toc(tstart);
accuracy(3) = norm(expm(i*H)-U);

end
\end{verbatim}

\smallskip \rule[0.5ex]{1\columnwidth}{1pt}   

\begin{verbatim}
% The Paige - Van Loan algorithm
function [Q,D] = PVL(V);

n = size(V);
n = n(1)/2;

Q = eye(2*n);
D = V; % Expect to keep V = Q*D*Q'

for k=1:(n-1)
	%Householder to fix most of the bottom half in column k
	if k<(n-1)
		[v, beta] = gallery('house',D(k+1+n:2*n,k));
		%on left of D, top, then bottom
		D(k+1:n,1:2*n) = D(k+1:n,1:2*n) ...
		- (conj(beta)*conj(v))*(conj(v)'*D(k+1:n,1:2*n));
		D(k+1+n:2*n,1:2*n) = D(k+1+n:2*n,1:2*n) ...
		- (beta*v)*(v'*D(k+1+n:2*n,1:2*n));
		%on right of D, top, then bottom
		D(1:2*n,k+1:n) = D(1:2*n,k+1:n) ...
		- (beta*(D(1:2*n,k+1:n)*conj(v) )) *conj(v)';
		D(1:2*n,k+1+n:2*n) = D(1:2*n,k+1+n:2*n) ...
		- (conj(beta)*(D(1:2*n,k+1+n:2*n)*v)) *v';
		%on right of Q
		Q(1:2*n,k+1:n) = Q(1:2*n,k+1:n) ...
		- (beta*(Q(1:2*n,k+1:n)*conj(v) )) *conj(v)';
		Q(1:2*n,k+1+n:2*n) = Q(1:2*n,k+1+n:2*n) ...
		- (conj(beta)*(Q(1:2*n,k+1+n:2*n)*v)) *v';
	end

	%A symplectic Givens rotation to clear out D(k+1+n,k)
	[G,y] = planerot([D(k+1,k);D(k+1+n,k)]);
	%on left of D
	top_row = D(k+1,1:2*n);
	bottom_row = D(k+1+n,1:2*n);
	D(k+1,1:2*n) = G(1,1)*top_row + G(1,2)*bottom_row;
	D(k+1+n,1:2*n) = G(2,1)*top_row + G(2,2)*bottom_row;
	%on right of D
	left_col = D(1:2*n,k+1);
	right_col = D(1:2*n,k+1+n);
	D(1:2*n,k+1)=conj(G(1,1))*left_col+conj(G(1,2))*right_col;
	D(1:2*n,k+1+n)=conj(G(2,1))*left_col+conj(G(2,2))*right_col;
	%on right of Q
	left_col = Q(1:2*n,k+1);
	right_col = Q(1:2*n,k+1+n);
	Q(1:2*n,k+1)=conj(G(1,1))*left_col+conj(G(1,2))*right_col;
	Q(1:2*n,k+1+n)=conj(G(2,1))*left_col+conj(G(2,2))*right_col;

	%Householder to fix top half in column k
	if k<(n-1)
		[v, beta] = gallery('house',D(k+1:n,k));
		%on left of D, top, then bottom
		D(k+1:n,1:2*n) = D(k+1:n,1:2*n) ...
		- (beta*v)*(v'*D(k+1:n,1:2*n));
		D(k+1+n:2*n,1:2*n) = D(k+1+n:2*n,1:2*n) ...
		- (conj(beta)*conj(v))*(conj(v)'*D(k+1+n:2*n,1:2*n));
		%on right of D, top, then bottom
		D(1:2*n,k+1:n) = D(1:2*n,k+1:n) ...
		- (conj(beta)*(D(1:2*n,k+1:n)*v)) *v';
		D(1:2*n,k+1+n:2*n) = D(1:2*n,k+1+n:2*n) ...
		- (beta *(D(1:2*n,k+1+n:2*n)*conj(v)))*conj(v)';
		%on right of Q
		Q(1:2*n,k+1:n) = Q(1:2*n,k+1:n) ...
		- (conj(beta)*(Q(1:2*n,k+1:n)*v)) *v';
		Q(1:2*n,k+1+n:2*n) = Q(1:2*n,k+1+n:2*n) ...
		- (beta*(Q(1:2*n,k+1+n:2*n)*conj(v)))*conj(v)';
	end
end

end
\end{verbatim}

\smallskip \rule[0.5ex]{1\columnwidth}{1pt}                    

\begin{verbatim}
function Y = dual(X)

N = size(X);
N = N(1)/2;

Y(N+1:2*N,N+1:2*N) = X(1:N,1:N).';
Y(1:N,1:N) = X(N+1:2*N,N+1:2*N).';
Y(1:N,N+1:2*N) = - X(1:N,N+1:2*N).';
Y(N+1:2*N,1:N) = - X(N+1:2*N,1:N).';

end
\end{verbatim}

\rm 


\begin{thebibliography}{10}

\bibitem{anderson1999lapack}
{\sc E.~Anderson, Z.~Bai, and C.~Bischof}, {\em LAPACK Users' guide}, vol.~9,
  Society for Industrial Mathematics, 1999.

\bibitem{benner-skew}
{\sc P.~Benner, D.~Kressner, and V.~Mehrmann}, {\em {Skew-Hamiltonian and
  Hamiltonian eigenvalue problems: Theory, algorithms and applications}}, in
  Proceedings of the Conference on Applied Mathematics and Scientific
  Computing, Springer, 2005, pp.~3--39.

\bibitem{bunseStructuredWigenvalue}
{\sc A.~Bunse-Gerstner, R.~Byers, and V.~Mehrmann}, {\em {A chart of numerical
  methods for structured eigenvalue problems}}, SIAM Journal on Matrix Analysis
  and Applications, 13 (1992), p.~419.

\bibitem{cheng2000return}
{\sc S.~Cheng, N.~Higham, C.~Kenney, and A.~Laub}, {\em Return to the middle
  ages: A half-angle iteration for the logarithm of a unitary matrix}, in
  Proceedings of the Fourteenth International Symposium of Mathematical Theory
  of Networks and Systems, Perpignan, France. CD ROM, 2000.
  
\bibitem{davidson2010essentially}
{\sc K.~Davidson}, {\em Essentially normal operators}, A Glimpse at Hilbert
  Space Operators,  (2010), pp.~209--222.

\bibitem{DavHighamMatrixFunctions2003}
{\sc P.~Davies and N.~Higham}, {\em A {S}chur-{P}arlett algorithm for computing
  matrix functions}.

\bibitem{fulga2011scattering}
{\sc I.~Fulga, F.~Hassler, and A.~Akhmerov}, {\em Scattering theory of
  topological insulators and superconductors}, Arxiv preprint arXiv:1106.6351,
  (2011).

\bibitem{fulga2011Wire}
{\sc I.~Fulga, F.~Hassler, A.~Akhmerov, and C.~Beenakker}, {\em Scattering
  formula for the topological quantum number of a disordered multimode wire},
  Physical Review B, 83 (2011), p.~155429.

\bibitem{hastings2001eigenvalue}
{\sc M.~Hastings}, {\em {Eigenvalue Distribution In The Self-Dual Non-Hermitian
  Ensemble}}, Journal of Statistical Physics, 103 (2001), pp.~903--913.

\bibitem{HastLorTheoryPractice}
{\sc M.~B. Hastings and T.~A. Loring}, {\em Topological insulators and {$C\sp
  *$}-algebras: Theory and numerical practice}, Ann. Physics, 326 (2011),
  pp.~1699--1759.

\bibitem{HenriciNonNormal1962}
{\sc P.~Henrici}, {\em Bounds for iterates, inverses, spectral variation and
  fields of values of non-normal matrices}, 4 (1962), pp.~24--40.

\bibitem{higham1994parallel}
{\sc N.~Higham and P.~Papadimitriou}, {\em {A parallel algorithm for computing
  the polar decomposition}}, Parallel Computing, 20 (1994), pp.~1161--1173.

\bibitem{katan2013modulated}
{\sc Y.~Katan and D.~Podolsky}, {\em Modulated floquet topological insulators},
  Physical Review Letters, 110 (2013), p.~016802.

  
\bibitem{lindner2011floquet}
{\sc N.~Lindner, G.~Refael, and V.~Galitski}, {\em Floquet topological
  insulator in semiconductor quantum wells}, Nature Physics, 7 (2011),
  pp.~490--495.

  


\bibitem{LoringQuantKth}
{\sc T.~A. Loring}, {\em Quantitative $k$-theory and spin chern numbers}.
\newblock arxiv:1302.0349.

\bibitem{LoringQuaternions}
\leavevmode\vrule height 2pt depth -1.6pt width 23pt, {\em Factorization of
  matrices of quaternions}, Exposition. Math., 30 (2012), pp.~250--267.

\bibitem{LorHastHgTe}
{\sc T.~A. Loring and M.~B. Hastings}, {\em Disordered topological insulators
  via {$C\sp *$}-algebras}, Europhys. Lett. EPL, 92 (2010), p.~67004.
  
\bibitem{moler1978nineteen}
{\sc C.~Moler and C.~Van~Loan}, {\em Nineteen dubious ways to compute the
  exponential of a matrix}, SIAM review, 20 (1978), pp.~801--836.

\bibitem{teo2010topological}
{\sc J.~Teo and C.~Kane}, {\em Topological defects and gapless modes in
  insulators and superconductors}, Physical Review B, 82 (2010), p.~115120.

\end{thebibliography}
\end{document}